\documentclass{amsart}
\usepackage{ms_arxiv}

 \title[Local wellposedness for 3-d periodic NLS] 
 {Local wellposedness for the critical nonlinear Schr\"odinger equation on $\mathbb{T}^3$}

\author[Gyu Eun Lee]{}

\thanks{This work was supported by NSF grants DMS-1265868, DMS-1600942 (principal investigator: Rowan Killip) and DMS-1500707 (principal investigator: Monica Vi\c{s}an).}

\begin{document}

\maketitle

\centerline{\scshape Gyu Eun Lee}

\bigskip

\begin{abstract}
    For $p\geq 2$, we prove local wellposedness for the nonlinear Schr\"odinger equation $(i\pt_t + \Delta)u = \pm|u|^pu$ on $\mathbb{T}^3$ with initial data in $H^{s_c}(\mathbb{T}^3)$, where $\mathbb{T}^3$ is a rectangular irrational $3$-torus and $s_c = \frac{3}{2} - \frac{2}{p}$ is the scaling-critical regularity.
    This extends work of earlier authors on the local Cauchy theory for NLS on $\mathbb{T}^3$ with power nonlinearities where $p$ is an even integer.
\end{abstract}

\section{Introduction}

In this paper we are concerned with the local theory for the Cauchy problem for the nonlinear Schr\"odinger equation (NLS) on a rectangular irrational $3$-torus $\bb{T}^3$.
For notational simplicity, we regard $\bb{T}^3$ as the set $\bb{R}^3/\bb{Z}^3$, and embed the irrationality of the torus into our choice of Riemannian metric.
We denote the corresponding Laplace-Beltrami operator by $\Delta$; see \citelist{\cite{GuOhWa13} \cite{KiVi16}}.
Under these conventions, the equation takes the form
\begin{equation}\label{eqn:NLS}
    \begin{cases}
        (i\pt_t + \Delta)u \pm |u|^p u = 0, & (t,x)\in\bb{R}\times\bb{T}^3,\\
        u(0,x) = u_0(x) \in H^{s_c}(\bb{T}^3), & s_c = \frac{3}{2} - \frac{2}{p}.
    \end{cases}
\end{equation}
We restrict our attention to nonlinearities with $p\geq 2$.
This includes the nonlinearities that are most relevant for physical problems, namely $p=2$ (cubic NLS) and $p=4$ (quintic NLS).
We consider this problem to be posed at critical regularity by analogy with the Euclidean case, for which the $\dot{H}^{s_c}(\bb{R}^3)$ norm is invariant under the NLS scaling symmetry $u_\lambda(t,x) = \lambda^{-2/p} u(\lambda^{-2}t,\lambda^{-1}x)$.
The choice of $+$ sign in \eqref{eqn:NLS} is called \ita{focusing}, and the $-$ sign is called \ita{defocusing}.
The focusing or defocusing nature of the equation typically has important implications in the global Cauchy theory of equations of this type; however, in this paper we are concerned only with local theory.

Our main result is:

\begin{main}[Local wellposedness]\label{theorem:main}
    Fix $p\geq 2$.
    Let $u_0\in H^{s_c}(\bb{T}^3)$.
    Then there exists a time of existence $T = T(u_0)$ and a unique solution $u\in C_t([0,T);H^{s_c}(\bb{T}^3))\cap X^{s_c}([0,T))$ to \eqref{eqn:NLS}.
\end{main}
Here $X^s$ refers to the adapted function spaces developed by Koch, Tataru, and collaborators \citelist{\cite{HeTaTz11} \cite{KoTaVi14}}.

Equation \eqref{eqn:NLS} is the $3$-dimensional case of the general periodic $H^s$-critical NLS with power nonlinearity:
\begin{equation}\label{eqn:NLS_general}
    \begin{cases}
        (i\pt_t + \Delta)u \pm |u|^p u = 0, & (t,x)\in\bb{R}\times\bb{T}^d,\\
        u(0,x) = u_0(x) \in H^{s_c}(\bb{T}^d), & s_c = \frac{d}{2} - \frac{2}{p}.
    \end{cases}
\end{equation}
Theorem \ref{theorem:main} extends earlier results on the Cauchy theory of Equation \eqref{eqn:NLS_general}.
Work to date has focused on the $H^1$-critical problem, i.e. $p = \frac{4}{d-2}$.
The present paper is an application of two lines of earlier research: adapted function spaces for critical problems, and scale-invariant Strichartz estimates for the Schr\"odinger equation on tori.
Here is a brief survey of such results and related works:
\begin{enumerate}
    \item The earliest scale-invariant Strichartz estimates for the Schr\"odinger equation on tori are due to Bourgain. In \cite{Bo93.Schrodinger}, he obtained a range of scale-invariant Strichartz estimates on square tori, which were applied to the local and small-data global theory for periodic NLS.
    \item The critical function spaces $X^s$ and $Y^s$ were introduced by Herr, Tataru, and Tzvetkov in \cite{HeTaTz12}, in which they were used to prove local wellposedness for the $H^1$-critical NLS on a partially irrational torus with $d=4$.
    This built off earlier work by Hadac, Herr, and Koch \cite{HaHeKo09}, in which the atomic spaces $U^p$ and $V^p$ (precursors to $X^s$ and $Y^s$) were introduced, and by Herr, Tataru, and Tzvetkov \cite{HeTaTz11}, which established local wellposedness and small-data global wellposedness for $H^1$-critical NLS on a square torus with $d=3$.
    Ionescu and Pausader extended the result of \cite{HeTaTz12} to obtain large-data global wellposedness $H^1$-critical defocusing NLS on $\bb{T}^3$ \cite{IoPa12}.
    \item In \cite{GuOhWa13}, Guo, Oh, and Wang proved a range of new scale-invariant Strichartz estimates for the linear Schr\"odinger evolution on irrational tori.
    As an application, they extended the result of \cite{HeTaTz11} to a partially irrational torus with $d=3$.
    They also established local wellposedness for \eqref{eqn:NLS_general} for $d\in\{2,3,4\}$ and certain choices of $p = 2k$, $k\in\bb{N}$: $d=2$ and $k\geq 6$; $d=3$ and $k\geq 3$; and $d\geq 4$ and $k\geq 2$.
    Strunk extended this result to $d=2, k\geq 3$ and $d=3, k=2$, using multilinear Strichartz estimates \cite{Strunk14} .
    \item In \cite{BoDe15}, Bourgain and Demeter proved the full range of scale-invariant Strichartz estimates on square tori.
    Killip and Visan extended this result to rectangular irrational tori, and used them to prove local wellposedness for $H^1$-critical NLS for such tori with $d=3,4$ \cite{KiVi16}.
    One feature of the proof was the use of a bilinear Strichartz estimate (which had appeared earlier in \cite{HeTaTz12} in the $d=4$ case), which allowed for a simpler proof than previous approaches which relied on multilinear estimates with three or more factors.
\end{enumerate}
In terms of methodology, our proof of Theorem \ref{theorem:main} is closely in line with that of \cite{KiVi16}, in that we rely mainly on bilinear Strichartz estimates.

Here is a brief outline of our paper.
In Section \ref{sec:notation} we lay out the notation and preliminary results used in the rest of the paper.
Our proof of Theorem \ref{theorem:main} goes via a standard contraction mapping argument for the Duhamel operator
\[
    \Phi(u)(t) = e^{it\Delta}u_0 - i\int_0^t e^{i(t-s)\Delta}F(u)~ds,
\]
where $e^{it\Delta}$ denotes the Schr\"odinger propagator and $F(u) = \pm|u|^p u$.
Our goal is to establish that $\Phi$ is a self-map and contraction mapping on some ball in some appropriate function space.
This argument is detailed in Section \ref{sec:contraction}; the proof of the contraction estimate within comprises the bulk of our paper.
Once the contraction estimate is established, we may arrive at Theorem \ref{theorem:main} via arguments from previous works, such as that in \cite{KiVi16}.
Our proof of the contraction estimate breaks into cases depending on the value of $p$.
For values of $p$ that are even integers, the term $F(u+w)-F(u)$ arising in the Duhamel integral can be manipulated algebraically, and we can straightforwardly adapt earlier proofs such as those given in \cites{HeTaTz11,IoPa12,KiVi16}; we exhibit this argument explicitly for the cubic NLS.
For other values of $p$, we apply a paradifferential calculus technique of Bony to express $F(u+w)-F(u)$ in a form similar to what can be done in the case of $p$ even.

Lastly, we note that local wellposedness is conjectured to hold for all nonlinearities $\pm|u|^p u$ with $p > \frac{4}{3}$, i.e. those with $s_c>0$.
This is the largest range of $p$ for which scale-invariant Strichartz estimates allow us to control a solution $u$ in $L_{t,x}^q$ by $s_c$ derivatives of $u$ (see Theorem \ref{theorem:scale_invariant_Strichartz}).
The methods in this paper only establish the result down to $p\geq 2$.
This is for purely technical reasons, and we have yet to undertake an investigation into ways to circumvent the current limitations of our method.
These limitations are the following.
First, in the case of non-even powers $p$ we make use of an estimate (Lemma \ref{lem:Strichartz_derivatives}) which only holds for $p>2$.
Second, we use the fact that for $p>2$ the function $F(z) = |z|^pz$ admits at least three derivatives in order to perform three iterations of a paradifferential linearization technique.
We recover the case $p=2$ manually, using algebraic simplifications that arise from the even power.
We do not believe that the limitations of our methods indicate any serious barriers to pushing the local wellposedness threshold down to $p>\frac{4}{3}$.

\section{Notation and preliminary results}\label{sec:notation}

We define the symmetric spacetime norms
\[
    \|u\|_{L_{t,x}^p([0,T)\times\bb{T}^3)}
        = \bigg(\int_0^T \int_{\bb{T}^3} |u(t,x)|^p ~dxdt\bigg)^\frac{1}{p}
\]
where $1\leq p\leq\infty$, with obvious changes if $p=\infty$.
We often suppress the spacetime domain and write $\|u\|_{L_{t,x}^p}$.

We write $X\lesssim Y$ to denote $X \leq CY$, where $C$ is some constant; when we wish to indicate dependence of $C$ on parameters $z_1,\ldots,z_k$, we write $X \lesssim_{z_1,\ldots,z_k} Y$.

Let $\phi:\bb{R}^d\to [0,\infty)$ be a smooth radial cutoff with $\phi(x) = 1$ for $|x|\leq 1$ and $\phi(x) = 0$ for $x\geq 2$.
We take the convention $0\in\bb{N}$.
For $N\in 2^\bb{N}$ a dyadic integer, write $\phi_N(x) = \phi(x/N)$ and $\psi_N(x) = \phi_N(x) - \phi_{\frac{N}{2}}(x)$, with the convention $\psi_1(x) = \phi(x)$.
For $f:\bb{T}^d\to\bb{C}$ with Fourier coefficients $\wh{f}(\xi)$, $\xi\in\bb{Z}^d$, we define the Littlewood-Paley projections as Fourier multipliers:
\[
    \wh{u_{\leq N}}(\xi) = \wh{P_{\leq N}f}(\xi) = \phi_N(\xi)\wh{f}(\xi), 
    ~\wh{u_N}(\xi) = \wh{P_Nf}(\xi) = \psi_N(\xi)\wh{f}(\xi), ~\xi\in\bb{Z}^d.
\]
In what follows, integers denoted by $N$, $N_0, N_1$, and the like will be implicitly assumed to be dyadic integers.

We now define the adapted function spaces $U^p$, $V^p$, $X^s$, and $Y^s$.
We restrict ourselves to stating the definitions and basic properties.
For a more complete reference, see \cite{KoTaVi14}.
Fix a finite time interval $[0,T)$.
Let $H$ be a separable Hilbert space over $\bb{C}$; for our purposes, this will be $\bb{C}$, $L^2(\bb{T}^3)$, or $H^{s_c}(\bb{T}^3)$.
Let $\mcal{Z}$ be the set of finite partitions $0 = t_0 < t_1 < \cdots < t_K \leq T$.
We adopt the convention that $v(T) = 0$ for all functions $v:[0,T)\to H$.
\begin{definition}
    Let $1 \leq p < \infty$.
    A \ita{$U^p$-atom} is a function $a:[0,T)\to H$ of the form
    \[
        a = \sum_{k=0}^{K-1} 1_{[t_k,t_k+1)}\phi_k,
    \]
    where $\{t_k\}\in\mcal{Z}$ and $\{\phi_k\}\subset H$ with $\sum_{k=0}^{K-1} \|\phi_k\|_H^p \leq 1$.
    The space $U^p([0,T);H)$ is the space of all functions $u:[0,T)\to H$ admitting a representation of the form
    \[
        u = \sum_{j=1}^\infty \lambda_j a_j,    
    \]
    where $a_j$ are $U^p$-atoms and $(\lambda_j) \in \ell^1(\bb{C})$.
    $U^p([0,T);H)$ is a Banach space under the norm
    \[
        \|u\|_{U^p} 
            = \inf \bigg\{ \sum_{j=1}^\infty |\lambda_j| : u = \sum_{j=1}^\infty \lambda_ja_j ~\tnm{with}~ (\lambda_j)\in\ell^1(\bb{C}) ~\tnm{and}~ a_j ~U^p-\tnm{atoms}\bigg\}.    
    \]
\end{definition}
\begin{definition}
    Let $1\leq p < \infty$.
    $V^p([0,T);H)$ is the space of all functions $v:[0,T)\to H$ with finite $V^p$-seminorm $\|v\|_{V^p}$, where
    \[
        \|v\|_{V^p}
            = \sup_{\{t_k\}\in \mcal{Z}} \bigg( \sum_{k=1}^{K-1} \|v(t_k) - v(t_{k-1})\|_H^p \bigg)^\frac{1}{p}.
    \]
    The space $V_{\tnm{rc}}^p$ is the subspace of $V^p$ consisting of right-continuous functions in $V^p$, normalized so that $\lim_{t\to 0^+} v(t) = 0$.
    The $V^p$-seminorm restricts to a norm on $V_{\tnm{rc}}^p$, and $V_{\tnm{rc}}^p$ is a Banach space under this norm.
\end{definition}
\begin{definition}
    Let $s\in\bb{R}$, $d\geq 1$.
    We define $X^s([0,T))$ and $Y^s([0,T))$ to be the Banach spaces of all functions $u:[0,T)\to H^s(\bb{T}^d)$ such that for every $\xi\in\bb{Z}^d$, the map $t\mapsto \wh{e^{-it\Delta}u(t)}(\xi)$ is in $U^2([0,T);\bb{C})$ and $V_{\tnm{rc}}^2([0,T);\bb{C})$ respectively, with norms
    \[
        \|u\|_{X^s([0,T))} = \bigg( \sum_{\xi\in\bb{Z}^d} \langle\xi\rangle^{2s}\|\wh{e^{-it\Delta}u}(t)(\xi)\|_{U^2}^2 \bigg)^\frac{1}{2},
    \]
    \[
        \|u\|_{Y^s([0,T))} = \bigg( \sum_{\xi\in\bb{Z}^d} \langle\xi\rangle^{2s}\|\wh{e^{-it\Delta}u}(t)(\xi)\|_{V^2}^2 \bigg)^\frac{1}{2}.
    \]
\end{definition}
Again, we will typically suppress the spacetime domain in our notation when it is obvious.
$X^s$ and $Y^s$ have a dual pairing in the following sense:
\begin{proposition}[$X^s$-$Y^s$ duality; \cite{HeTaTz11}*{Proposition 2.11}]\label{prop:XsYs_duality}
    Let $s\geq 0$ and $T>0$.
    For $f\in L^1([0,T);H^s(\bb{T}^d))$ we have
    \[
        \bigg\|\int_0^t e^{i(t-s)\Delta} f(s)~ds\bigg\|_{X^s([0,T))}
            \leq \sup_{\|v\|_{Y^{-s}([0,T)) = 1}} \bigg| \int_0^T\int_{\bb{T}^3} f(t,x)\ol{v(t,x)}~dxdt \bigg|.    
    \]
\end{proposition}
\begin{remark}\label{rem:XsYs_embeddings}
    We have a continuous embedding $X^s \incl Y^s$.
    We also have
    \[
        \|u\|_{L_t^\infty H_x^s} \lesssim \|u\|_{X^s},
    \]
    \[
        \bigg\| \int_0^t e^{i(t-s)\Delta}F(s)~ds\bigg\|_{X^s} \lesssim \|F\|_{L_t^1H_x^s}.    
    \]
\end{remark}
\begin{remark}
    The spaces $X^s$ and $Y^s$ have the scaling of $L_t^\infty H_x^s$ and enjoy several of its Fourier-based properties: for instance, we have
    \[
        \|P_N u\|_{Y^s} \sim N^s\|P_Nu\|_{Y^0}
    \]
    and
    \[
        \|u\|_{Y^s} = \bigg(\sum_N \|P_N u\|_{Y^s}^2 \bigg)^\frac{1}{2}.
    \]
\end{remark}
We now state the main tools of our analysis.
\begin{theorem}[Strichartz estimates \cite{KiVi16}]\label{theorem:scale_invariant_Strichartz}
    Fix $d\geq 1$, $1 \leq N \in 2^\bb{N}$, and $p > \frac{2(d+2)}{d}$.
    Then
    \[
        \|P_Cu\|_{L_{t,x}^p([0,1]\times\bb{T}^d)}
            \lesssim N^{\frac{d}{2}-\frac{d+2}{p}}\|P_Cu\|_{Y^0(\bb{T}^d)}
    \]
    for all $p> \frac{2(d+2)}{d}$, where $C\subset\bb{R}^d$ is a cube of side length $N$ and $P_C$ denotes the Fourier projection to $C$.
\end{theorem}
As a direct consequence of Theorem \ref{theorem:scale_invariant_Strichartz}, we have:
\begin{lemma}[Bilinear Strichartz estimate \cite{KiVi16}*{Lemma 3.1}]\label{lem:bilinear_Strichartz}
    Fix $d\geq 3$ and $T\leq 1$.
    Then for all $1 \leq N_2 \leq N_1$,
    \[
        \|u_{N_1}v_{N_2}\|_{L_{t,x}^2}
            \lesssim N_2^{\frac{d-2}{2}}\|u_{N_1}\|_{Y^0}\|v_{N_2}\|_{Y^0}.    
    \]
    The implicit constant does not depend on $T$.
\end{lemma}
Lastly, we state some useful fractional calculus estimates:
\begin{proposition}[Fractional product rule]\label{prop:frac_product}
    Let $d\geq 1$, $s>0$, $1 < p < \infty$, and $1 < p_2,q_2\leq \infty$ such that $\frac{1}{p} = \frac{1}{p_1} + \frac{1}{p_2} = \frac{1}{q_1} + \frac{1}{q_2}$.
    Then
    \[
        \||\nabla|^s(fg)\|_{L^p(\bb{T}^d)}
            \lesssim \||\nabla|^sf\|_{L^{p_1}(\bb{T}^d)}\|g\|_{L^{p_2}(\bb{T}^d)} + \||\nabla|^sg\|_{L^{q_1}(\bb{T}^d)}\|f\|_{L^{q_2}(\bb{T}^d)}.
    \]
\end{proposition}
\begin{proposition}[Fractional chain rule]\label{prop:frac_chain}
    Suppose $F:\bb{C}\to\bb{C}$ satisfies $|F(u)-F(v)| \lesssim |u-v|(G(u)+G(v))$ for some $G:\bb{C}\to[0,\infty)$.
    Let $d\geq 1$, $0<s <1$, $1 < p <\infty$, and $1 < p_2 \leq \infty$, such that $\frac{1}{p} = \frac{1}{p_1} + \frac{1}{p_2}$.
    Then
    \[
        \||\nabla|^s F(u)\|_{L^p(\bb{T}^d)}
            \lesssim \||\nabla|^s u\|_{L^{p_1}(\bb{T}^d)}\|G(u)\|_{L^{p_2}(\bb{T}^d)}.    
    \]
\end{proposition}
\begin{proposition}[Nonlinear Bernstein]\label{prop:nonlinear_Bernstein}
    Let $G:\bb{C}\to\bb{C}$ be H\"older continuous of order $\alpha\in(0,1]$.
    Let $d\geq 1$ and $1 \leq p \leq \infty$.
    Then for $u:\bb{T}^d\to\bb{C}$ smooth and periodic, we have
    \[
        \|P_N G(u)\|_{L^{p/\alpha}(\bb{T}^d)} \lesssim N^{-\alpha}\|\nabla u\|_{L^p(\bb{T}^d)}^\alpha  
    \]
    for all $N>1$.
\end{proposition}
In the Euclidean setting Propositions \ref{prop:frac_product} and \ref{prop:frac_chain} are due to Christ and Weinstein \cite{ChWe91}.
Proposition \ref{prop:nonlinear_Bernstein} in the Euclidean setting appears in \cite{KiVi13.Clay}.
These results can be extended to the periodic setting via estimates on the periodic Littlewood-Paley convolution kernels.

\section{The contraction mapping estimate}\label{sec:contraction}
Fix an initial datum $u_0\in H^{s_c}(\bb{T}^3)$.
Consider the Duhamel operator
\[
    \Phi(u)(t) = e^{it\Delta}u_0 - i\int_0^t e^{i(t-s)\Delta}F(u(s))~ds,
\]
where $F(u) = \pm |u|^p u$.
As the choice of sign is irrelevant for everything that follows, we will take $F(u) = |u|^p u$ from here on.
To prove Theorem \ref{theorem:main} for the datum $u_0$ it suffices to show that there exists a time $T$ and a ball $B \subset C_t([0,T);H^{s_c}(\bb{T}^3))\cap X^{s_c}([0,T))$ on which $\Phi$ is a self-map and contraction mapping: then the Banach fixed-point theorem implies that $\Phi$ has a unique fixed point in $B$, which is the solution to \eqref{eqn:NLS} we seek.
The goal of this section is the contraction mapping estimate:
\begin{proposition}\label{prop:main_estimates}
    Fix $p \geq 2$.
    Let $0 < T \leq 1$.
    Then
    \begin{align*}
        \bigg\| \int_0^t e^{i(t-s)\Delta} &[F(u+w)(s) - F(u)(s)]~ds \bigg\|_{X^{s_c}([0,T))}\\
            &\lesssim \|w\|_{X^{s_c}([0,T))}(\|u\|_{X^{s_c}([0,T))} + \|w\|_{X^{s_c}([0,T))})^p.
    \end{align*}
    The implicit constant does not depend on $T$.
\end{proposition}
Proposition \ref{prop:main_estimates}, together with Propositions \ref{prop:frac_product} and \ref{prop:frac_chain}, imply via arguments by previous authors that $\Phi$ is indeed a self-map and contraction mapping on some ball in $X^{s_c}([0,T])\cap C_tH_x^{s_c}([0,T]\times\bb{T}^3)$, provided that $T$ is chosen sufficiently small. As just one example of this argument, we refer the reader to Section 4 of \cite{KiVi16} for the details in the case of the $H^1$-critical cubic and quintic NLS.
Therefore, we focus our attention on the proof of Proposition \ref{prop:main_estimates} for the remainder of this paper.

First we make some reductions.
By $X^s$-$Y^s$-duality (Proposition \ref{prop:XsYs_duality}) and self-adjointness of Littlewood-Paley projections,
\begin{align*}
    \bigg\| \int_0^t e^{i(t-s)\Delta} &P_{\leq N}[F(u+w)(s) - F(u)(s)]~ds \bigg\|_{X^{s_c}([0,T))}\\
        &\leq \sup_{\|\tilde{v}\|_{Y^{-s_c}=1}} \bigg| \int_0^T\int_{\bb{T}^3} P_{\leq N}[F(u+w)(t) - F(u)(t)]\ol{\tilde{v}(t,x)}~dxdt\bigg|\\
        &= \sup_{\|\tilde{v}\|_{Y^{-s_c}=1}} \bigg| \int_0^T\int_{\bb{T}^3} [F(u+w)(t) - F(u)(t)]\ol{P_{\leq N}\tilde{v}(t,x)}~dxdt\bigg|
\end{align*}
We will prove the following estimate:
\begin{align*}\label{eqn:reduced_main_estimates}
    \bigg| \int_0^T\int_{\bb{T}^3} v(t,x)&[F(u+w)(t) - F(u)(t)]~dxdt\bigg|\\
        &\lesssim \|v\|_{Y^{-s_c}}\|w\|_{Y^{s_c}}(\|u\|_{Y^{s_c}} + \|w\|_{Y^{s_c}})^p. \numberthis
\end{align*}
Taking $v = \ol{P_{\leq N}\tilde{v}}$ and letting $N\to\infty$ then gives us Proposition \ref{prop:main_estimates}.

Let us give a brief outline of the proof of \eqref{eqn:reduced_main_estimates}.
This estimate is easiest to prove when $p$ is an even integer.
In this case, expanding $u$ and $w$ in terms of their Littlewood-Paley decompositions, we can write $F(u+w) - F(u)$ as a sum of products of frequency projections of $u,\ol{u},w,\ol{w}$, and emulate earlier arguments such as those in \cites{HeTaTz11,IoPa12,KiVi16}.
We will provide an explicit argument for $p=2$ in Section \ref{sec:cubic}.
This is also necessary, as some of our estimates for the case $p>2$ do not extend down to the endpoint.

The remainder of the paper treats $p>2$.
When $p$ is not an even integer, the above argument can no longer be carried out exactly.
However, the nonlinearity can be rewritten in a more manageable form using a paradifferential calculus technique known as the Bony linearization formula: we write $F(u) = F(u_{\leq 1}) + \sum_{N\geq 2} [F(u_{\leq N}) - F(u_{\leq\frac{N}{2}})]$, and use the fundamental theorem of calculus to rewrite the differences, obtaining expressions essentially of the form $F(u)\sim \sum_{N\geq 1} u_N|u_{\leq N}|^p$.
Iterating this sort of argument gives us an expression for $F(u+w) - F(u)$ that is essentially multilinear and treatable with known technology.

\subsection{Proof of Theorem \ref{theorem:main} for the cubic NLS}\label{sec:cubic}

We first consider the $H^{\frac{1}{2}}$-critical cubic NLS, the case $p=2$ of Theorem \ref{theorem:main}.
This proof is very similar to the proof of the contraction estimate for the $H^1$-critical cubic NLS on $\bb{T}^4$ in \cite{KiVi16}.
By the above discussion, to prove Theorem \ref{theorem:main} for $p=2$ it suffices to prove \eqref{eqn:reduced_main_estimates} for $p=2$.

\begin{proof}[Proof of \eqref{eqn:reduced_main_estimates} for $p=2$]
    Expanding $u$ and $w$ by their Littlewood-Paley decompositions and doing some combinatorics, \eqref{eqn:reduced_main_estimates} reduces to an estimate of the form
    \begin{equation}\label{eqn:cubic_main_estimate}
        \sum_{N_0 \geq 1} \sum_{N_1 \geq N_2 \geq N_3\geq 1} \bigg| \int_0^T \int_{\bb{T}^3} v_{N_0}u_{N_1}^{(1)}u_{N_2}^{(2)}u_{N_3}^{(3)} ~dxdt\bigg|
            \lesssim \|v\|_{Y^{-\frac{1}{2}}}\prod_{j=1}^3 \|u^{(j)}\|_{X^{\frac{1}{2}}},
    \end{equation}
    where $u^{(j)}$ are chosen from $\{u,\ol{u},w,\ol{w}\}$.
    In order for the integrals in \eqref{eqn:cubic_main_estimate} to be nonzero, the two highest frequencies must be comparable.
    Therefore the sum splits into two cases:
    \begin{enumerate}[1.]
        \item $N_0 \sim N_1 \geq N_2 \geq N_3$.
        We use the following idea which appeared in \cites{HeTaTz11,KiVi16}.
        Let $\bb{Z}^3 = \bigcup_j C_j$ be a partition of frequency space $\bb{Z}^3$ into cubes $C_j$ of side length $N_2$.
        We write $C_j \sim C_k$ if the sum set $C_j + C_k$ overlaps the Fourier support of $P_{\leq 2N_2}$.
        For a given $C_k$, there are a bounded number (independent of $k$ and $N_2$) of $C_j$ such that $C_j \sim C_k$.
        By H\"older, Strichartz, and summing via Cauchy-Schwarz, we estimate:
        \begin{align*}
            & \sum_{N_0 \sim N_1 \geq N_2 \geq N_3\geq 1} \bigg| \int_0^T \int_{\bb{T}^3} v_{N_0}u_{N_1}^{(1)}u_{N_2}^{(2)}u_{N_3}^{(3)} ~dxdt\bigg|\\
            &= \sum_{N_0\sim N_1 \geq N_2\geq N_3} \sum_{C_j\sim C_k} \left| \int_0^T\int_{\bb{T}^3} (P_{C_j}v_{N_0})(P_{C_k}u_{N_1}^{(1)})u_{N_2}^{(2)}u_{N_3}^{(3)}~dxdt\right| \\
                &\leq \sum_{N_0\sim N_1 \geq N_2\geq N_3} \sum_{C_j\sim C_k} \|P_{C_j}v_{N_0}\|_{L_{t,x}^{18/5}} \|P_{C_k}u_{N_1}^{(1)}\|_{L_{t,x}^{18/5}} \|u_{N_2}^{(2)}\|_{L_{t,x}^{18/5}} \|u_{N_3}^{(3)}\|_{L_{t,x}^6}\\
                &\lesssim \sum_{N_0\sim N_1 \geq N_2\geq N_3} \sum_{C_j\sim C_k} N_2^{\frac{1}{3}}N_3^{\frac{2}{3}}\|P_{C_j}v_{N_0}\|_{Y^0} \|P_{C_k}u_{N_1}^{(1)}\|_{Y^0} \|u_{N_2}^{(2)}\|_{Y^0} \|u_{N_3}^{(3)}\|_{Y^0}\\
                &\lesssim \sum_{N_0\sim N_1 \geq N_2\geq N_3} \sum_{C_j\sim C_k} \left(\frac{N_3}{N_2}\right)^\frac{1}{6}\|P_{C_j}v_{N_0}\|_{Y^{-\frac{1}{2}}} \|P_{C_k}u_{N_1}^{(1)}\|_{Y^{\frac{1}{2}}} \|u_{N_2}^{(2)}\|_{Y^{\frac{1}{2}}} \|u_{N_3}^{(3)}\|_{Y^{\frac{1}{2}}}\\
                &\lesssim \|v\|_{Y^{-\frac{1}{2}}}\prod_{j=1}^3 \|u^{(j)}\|_{Y^{\frac{1}{2}}}.
        \end{align*}
        From here \eqref{eqn:cubic_main_estimate} follows from the embedding $X^s\incl Y^s$.
        \item $N_0 \leq N_1 \sim N_2 \geq N_3$.
        For this sum we can estimate with just H\"older, Strichartz, and Cauchy-Schwarz:
        \begin{align*}
            &\sum_{N_0 \leq N_1 \sim N_2 \geq N_3\geq 1} \bigg| \int_0^T \int_{\bb{T}^3} v_{N_0}u_{N_1}^{(1)}u_{N_2}^{(2)}u_{N_3}^{(3)} ~dxdt\bigg|\\
                &\lesssim \sum_{N_0 \leq N_1 \sim N_2 \geq N_3\geq 1} \|v_{N_0}\|_{L_{t,x}^{18/5}}\|u_{N_1}^{(1)}\|_{L_{t,x}^{18/5}}\|u_{N_2}^{(2)}\|_{L_{t,x}^{18/5}}\|u_{N_3}^{(3)}\|_{6}\\
                &\lesssim \sum_{N_0 \leq N_1 \sim N_2 \geq N_3\geq 1} N_0^{\frac{1}{9}}N_1^{\frac{1}{9}}N_2^{\frac{1}{9}}N_3^{\frac{2}{3}}\|v_{N_0}\|_{Y^0}\|u_{N_1}^{(1)}\|_{Y^0}\|u_{N_2}^{(2)}\|_{Y^0}\|u_{N_3}^{(3)}\|_{Y^0}\\
                &\sim \sum_{N_0 \leq N_1 \sim N_2 \geq N_3\geq 1} \frac{N_0^{\frac{11}{18}}N_3^{\frac{1}{6}}}{N_1^{\frac{7}{18}}N_2^{\frac{7}{18}}}\|v_{N_0}\|_{Y^{-\frac{1}{2}}}\|u_{N_1}^{(1)}\|_{Y^{\frac{1}{2}}}\|u_{N_2}^{(2)}\|_{Y^{\frac{1}{2}}}\|u_{N_3}^{(3)}\|_{Y^{\frac{1}{2}}}\\
                &\lesssim \|v\|_{Y^{-\frac{1}{2}}}\|u\|_{Y^{\frac{1}{2}}}\sum_{N_1 \sim N_2} \left(\frac{N_1}{N_2}\right)^\frac{2}{9}\|u_{N_1}^{(1)}\|_{Y^{\frac{1}{2}}}\|u_{N_2}^{(2)}\|_{Y^{\frac{1}{2}}}\\
                &\lesssim \|v\|_{Y^{-\frac{1}{2}}}\prod_{j=1}^3 \|u^{(j)}\|_{Y^{\frac{1}{2}}}. \qedhere
        \end{align*}
    \end{enumerate}
\end{proof}
\begin{remark}
    Essentially the same proof will also give the contraction estimate for $p = 4,6,8,\ldots$.
    We can even take the same H\"older exponents for the four highest frequencies; any remaining lower-frequency terms can be placed in $L_{t,x}^\infty$, and a similar argument as above will establish the claim.
\end{remark}

\subsection{Paradifferential linearization}\label{sec:nonlinearity_decomposition}

We now extend the argument of the previous section to the case of general $p>2$.
Our main tool is the following paradifferential formula, which appeared in \cite{Bo81} with alternative hypotheses.
See \cite{Taylor00.tools} for a textbook treatment of this subject.
\begin{proposition}[Bony linearization formula]\label{prop:Bony_linearization}
    Let $g\in H^{s_c}(\bb{T}^d)$, $d\geq 3$, and let $F(z) = |z|^pz$ for $z\in\bb{C}$, $p\geq 0$.
    Then for all $1 \leq q < \frac{d}{2}$, we have
    \begin{equation}\label{eqn:LP_sum}
        F(g)
            = \lim_{N\to\infty} F(g_{\leq N}) 
            = \lim_{N\to\infty} F(g_{\leq 1}) + \sum_{2\leq M\leq N} [F(g_{\leq M}) - F(g_{\leq \frac{M}{2}})],
    \end{equation}
    where the limit is in the $L^q(\bb{T}^d)$-topology.
\end{proposition}
\begin{proof}
    By the bound $|F(g) - F(h)| \lesssim |g-h|(|g|^p + |h|^p)$ and Sobolev embedding, we have
    \begin{align*}
        \|F(g) - F(g_{\leq N})\|_{L^q}
            &\lesssim \|g - g_{\leq N}\|_{L^{dq/(d-2q)}}(\|g\|_{L^{dp/2}}^p + \|g_{\leq N}\|_{L^{dp/2}}^p)\\
            &\lesssim \|g - g_{\leq N}\|_{L^{dq/(d-2q)}}(\|g\|_{H^{s_c}}^p + \|g_{\leq N}\|_{H^{s_c}}^p).
    \end{align*}
    Under the given hypotheses, $1 < \frac{dq}{d-2q} < \infty$.
    Therefore $g_{\leq N} \to g$ in $L^{dq/(d-2q)}$.
    So in the limit, we obtain
    \begin{equation*}
        \lim_{N\to\infty} \|F(g) - F(g_{\leq N})\|_{L^q}
            \lesssim \|g\|_{H^{s_c}}^{1/p}\lim_{N\to\infty} \|g-g_{\leq N}\|_{L^{dq/(d-2q)}} = 0. \qedhere
    \end{equation*}
\end{proof}
We combine Proposition \ref{prop:Bony_linearization} and a linearization of $F(h_{\leq N}) - F(h_{\leq\frac{N}{2}})$.
Using the identity
\begin{equation}\label{eqn:FTC}
    F(u+w)-F(u) = w \int_0^1 \pt_z F(u+\theta w)~d\theta + \ol{w} \int_0^1 \pt_{\ol{z}} F(u+\theta w)~d\theta,
\end{equation}
we obtain
\begin{align*}\label{eqn:FTC_LP}
    F(u_{\leq N}) - F(u_{\leq \frac{N}{2}})
        &= u_N\int_0^1 \pt_z F((P_{\leq\frac{N}{2}} + \theta P_N)u)~d\theta\\
        &~ + \ol{u_N}\int_0^1 \pt_{\ol{z}} F((P_{\leq\frac{N}{2}} + \theta P_N)u)~d\theta. \numberthis
\end{align*}
The above expression is essentially of the form $F(u_{\leq N}) - F(u_{\leq \frac{N}{2}}) \sim u_N|u_{\leq N}|^p$ for the purpose of estimation, and thus \eqref{eqn:reduced_main_estimates} is morally equivalent to an estimate like
\begin{align*}
    \sum_{N_0\geq 1} \sum_{N_1\geq 1} \int_0^T\int_{\bb{T}^3} v_{N_0}&[(u_N + w_N)|u_{\leq N} + w_{\leq N}|^p - u_N|u_{\leq N}|^p]~dxdt\\
        &\lesssim \|v\|_{Y^{-s_c}}\|w\|_{Y^{s_c}}(\|u\|_{Y^{s_c}} + \|w\|_{Y^{s_c}})^p.
\end{align*}
We now make this precise.
Under the convention $g_{\leq\frac{1}{2}} = 0$, Proposition \ref{prop:Bony_linearization} implies the weak convergence statement 
\[
    \sum_{N\geq 1} \langle F(g_{\leq N}) - F(g_{\leq\frac{N}{2}}),\vp\rangle_{L^2}  = \langle F(g),\vp\rangle_{L^2}
\]
for all $\vp\in C(\bb{T}^d)$.
We now apply this to the integral appearing in \eqref{eqn:reduced_main_estimates}.
Noting that $v_{N_0}(t)\in C(\bb{T}^d)$ for all $v\in Y^{-s_c}([0,T))$ and $1\leq N_0\in 2^\bb{Z}$, we obtain
\begin{align*}
    \int_0^T\int_{\bb{T}^3} v(t,x)&[F(u+w) - F(u)](t,x)~dxdt\\
        &= \sum_{N_0\geq 1} \int_0^T\int_{\bb{T}^3} v_{N_0}(t,x)[F(u+w) - F(u)](t,x)~dxdt\\
        &= \sum_{N_0\geq 1}\sum_{N_1\geq 1} \int_0^T\int_{\bb{T}^3} v_{N_0}[(F(u_{\leq N_1}+w_{\leq N_1}) - F(u_{\leq \frac{N_1}{2}} + w_{\leq\frac{N_1}{2}})\\
        &\hspace{10em} - (F(u_{\leq N_1})-F(u_{\leq \frac{N_1}{2}}))](t,x)~dxdt.
\end{align*}
We split the inner sum into three regimes: $N_0 \gg N_1$, $N_0 \ll N_1$, and $N_0\sim N_1$, so that \eqref{eqn:reduced_main_estimates} reduces to the following:
\begin{proposition}\label{prop:incomparable_frequencies}
    Fix $p \geq 2$.
    Let $0 < T \leq 1$.
    Then
    \begin{align*}
        &\bigg|\sum_{N_0\geq 1}\sum_{N_1\gg N_0\geq 1} \int_0^T\int_{\bb{T}^3} v_{N_0}[(F(u_{\leq N_1}+w_{\leq N_1}) - F(u_{\leq \frac{N_1}{2}} + w_{\leq\frac{N_1}{2}}))\\
        &\hspace{12em} - (F(u_{\leq N_1})-F(u_{\leq \frac{N_1}{2}}))](t,x)~dxdt\bigg|\\
        &\hspace{12em}\lesssim \|v\|_{Y^{-s_c}}\|w\|_{Y^{s_c}}(\|u\|_{Y^{s_c}} + \|w\|_{Y^{s_c}})^p
    \end{align*}
    and
    \begin{align*}
        &\bigg|\sum_{N_0\geq 1}\sum_{1 \leq N_1 \ll N_0} \int_0^T\int_{\bb{T}^3} v_{N_0}[(F(u_{\leq N_1}+w_{\leq N_1}) - F(u_{\leq \frac{N_1}{2}} + w_{\leq\frac{N_1}{2}}))\\
        &\hspace{12em} - (F(u_{\leq N_1})-F(u_{\leq \frac{N_1}{2}}))](t,x)~dxdt\bigg|\\
        &\hspace{12em}\lesssim \|v\|_{Y^{-s_c}}\|w\|_{Y^{s_c}}(\|u\|_{Y^{s_c}} + \|w\|_{Y^{s_c}})^p.
    \end{align*}
\end{proposition}
\begin{proposition}\label{prop:comparable_frequencies}
    Fix $p \geq 2$.
    Let $0 < T \leq 1$.
    Then
    \begin{align*}
        &\bigg|\sum_{N_0\geq 1}\sum_{N_0\sim N_1 \geq 1} \int_0^T\int_{\bb{T}^3} v_{N_0}[(F(u_{\leq N_1}+w_{\leq N_1}) - F(u_{\leq \frac{N_1}{2}} + w_{\leq\frac{N_1}{2}}))\\
        &\hspace{12em} - (F(u_{\leq N_1})-F(u_{\leq \frac{N_1}{2}}))](t,x)~dxdt\bigg|\\
        &\hspace{12em}\lesssim \|v\|_{Y^{-s_c}}\|w\|_{Y^{s_c}}(\|u\|_{Y^{s_c}} + \|w\|_{Y^{s_c}})^p.
    \end{align*}
\end{proposition}
We treat each of these separately.
The integrals in Proposition \ref{prop:comparable_frequencies} can be controlled essentially as they are, while those in Proposition \ref{prop:incomparable_frequencies} are still difficult to estimate and require further linearization using Proposition \ref{prop:Bony_linearization} and Equation \eqref{eqn:FTC_LP}.

\subsection{Controlling sums over incomparable frequencies}
In this subsection we prove Proposition \ref{prop:incomparable_frequencies}.
We first express the integrands in Proposition \ref{prop:incomparable_frequencies} in a more manageable form.
Linearizing via \eqref{eqn:FTC_LP}, we write
\begin{align*}
    &(F(u_{\leq N}+w_{\leq N}) - F(u_{\leq \frac{N}{2}} + w_{\leq\frac{N}{2}})) - (F(u_{\leq N})-F(u_{\leq \frac{N}{2}}))\\
    &= w_N\int_0^1 \pt_z F((P_{\leq\frac{N}{2}} + \theta P_N)(u+w) ~d\theta\\
    &~~~~+ \ol{w_N}\int_0^1 \pt_{\ol{z}} F((P_{\leq\frac{N}{2}} + \theta P_N)(u+w))~d\theta\\
    &~~~~+u_N\int_0^1 [\pt_zF((P_{\leq\frac{N}{2}}+\theta P_N)(u+w)) - \pt_zF((P_{\leq\frac{N}{2}}+\theta P_N)u)]~d\theta\\
    &~~~~+ \ol{u_N}\int_0^1 [\pt_{\ol{z}}F((P_{\leq\frac{N}{2}}+\theta P_N)(u+w)) - \pt_{\ol{z}}F((P_{\leq\frac{N}{2}}+\theta P_N)u)]~d\theta.
\end{align*}
In the last two terms, we may linearize the difference again using \eqref{eqn:FTC}, obtaining:
\begin{align*}
    &u_N\int_0^1 [\pt_zF((P_{\leq\frac{N}{2}}+\theta P_N)(u+w)) - \pt_zF((P_{\leq\frac{N}{2}}+\theta P_N)u)]~d\theta\\
        &= u_N\int_0^1 (P_{\leq\frac{N}{2}}+\theta P_N)w \int_0^1 \pt_z^2F((P_{\leq\frac{N}{2}}+\theta P_N)(u+\eta w))]~d\eta d\theta\\
        &+ u_N\int_0^1 \ol{(P_{\leq\frac{N}{2}}+\theta P_N)w} \int_0^1 \pt_{\ol{z}}\pt_zF((P_{\leq\frac{N}{2}}+\theta P_N)(u+\eta w))]~d\eta d\theta
\end{align*}
and similarly for the term involving $\ol{u_N}$.
We summarize these calculations in the form
\begin{align*}\label{eqn:linearization_iteration_1}
    &(F(u_{\leq N_1}+w_{\leq N_1}) - F(u_{\leq \frac{N_1}{2}} + w_{\leq\frac{N_1}{2}})) - (F(u_{\leq N_1})-F(u_{\leq \frac{N_1}{2}}))\\
        &= w_{N_1}\int_0^1 \pt_z F((P_{\leq\frac{N_1}{2}} + \theta P_{N_1})(u+w) ~d\theta + ~\tnm{similar terms}\\
        &+ u_{N_1}\int_0^1 (P_{\leq\frac{N_1}{2}}+\theta P_N)w \int_0^1 \pt_z^2F((P_{\leq\frac{N_1}{2}}+\theta P_{N_1})(u+\eta w))]~d\eta d\theta\\
        &+ ~\tnm{similar terms}. \numberthis
\end{align*}
where by ``similar terms'' we indicate the same expression up to complex conjugates and conjugate derivatives in the appropriate places.

We now recall our heuristic: this expression is morally of the form
\begin{align*}
    &(F(u_{\leq N_1}+w_{\leq N_1}) - F(u_{\leq \frac{N_1}{2}} + w_{\leq\frac{N_1}{2}})) - (F(u_{\leq N_1})-F(u_{\leq \frac{N_1}{2}}))\\
        &\sim w_{N_1}\pt_zF(u_{\leq N_1} + w_{\leq N_1}) + u_{N_1}w_{\leq N_1}\pt_z^2F(u_{\leq N_1} + w_{\leq N_1})
\end{align*}
Inspired by this, we claim that Proposition \ref{prop:comparable_frequencies} follows from the following:
\begin{proposition}\label{prop:incomparable_frequencies_reduced}
    Fix $p \geq 2$.
    Let $0 < T \leq 1$.
    Then
    \begin{align*}
       \sum_{N_0\gg N_1 \geq 1} \bigg|\int_0^T\int_{\bb{T}^3} &v_{N_0}g_{N_1}D_{N_1}F(h_{\leq N_1})~dxdt\bigg|\\
            &\lesssim \|v\|_{Y^{-s_c}}\|w\|_{Y^{s_c}}\max\{\|g\|_{Y^{s_c}}, \|h\|_{Y^{s_c}}\}\|h\|_{Y^{s_c}}^{p-1}
    \end{align*}
    and
    \begin{align*}
        \sum_{N_0\ll N_1 \geq 1} \bigg|\int_0^T\int_{\bb{T}^3} &v_{N_0}g_{N_1}D_{N_1}F(h_{\leq N_1})~dxdt\bigg|\\
             &\lesssim \|v\|_{Y^{-s_c}}\|w\|_{Y^{s_c}}\max\{\|g\|_{Y^{s_c}}, \|h\|_{Y^{s_c}}\}\|h\|_{Y^{s_c}}^{p-1}.
     \end{align*}
    where $g \in\{u,w\}$, $D_{N_1}\in\{\pt_z,\pt_{\ol{z}}\}$ if $g=w$, and
    \[
        D_{N_1}\in\{w_{\leq N_1}\pt_z^2,w_{\leq N_1}\pt_z\pt_{\ol{z}},w_{\leq N_1}\pt_{\ol{z}}^2\}
    \]
    if $g = u$.
\end{proposition}
First, let us see how this simplified estimate leads to Proposition \ref{prop:incomparable_frequencies}.
\begin{proof}[Sketch of Proposition \ref{prop:incomparable_frequencies} assuming Proposition \ref{prop:incomparable_frequencies_reduced}]
    By our work to this point, Proposition \ref{prop:incomparable_frequencies} can be proved by establishing the corresponding estimate for each term in Equation \eqref{eqn:linearization_iteration_1}.
    Suppose, for instance, that we wish to prove the estimate
    \begin{align*}
        \bigg| \sum_{N_0\gg N_1} \int_0^T \int_{\bb{T}^3} &v_{N_0}w_{N_1}\int_0^1 \pt_z F((P_{\leq\frac{N_1}{2}} + \theta P_{N_1})(u+w)) ~d\theta dxdt\bigg|\\
            &\lesssim \|v\|_{Y^{-s_c}}\|w\|_{Y^{s_c}}(\|u\|_{Y^{s_c}} + \|w\|_{Y^{s_c}})^p.
    \end{align*}
    Applying Fubini, it is sufficient to show that
    \begin{align*}
        \sum_{N_0\gg N_1} \int_0^1\bigg(\int_0^T \int_{\bb{T}^3} &|v_{N_0}w_{N_1} \pt_z F((P_{\leq\frac{N_1}{2}} + \theta P_{N_1})(u+w))| ~dxdt\bigg)~d\theta\\
            &\lesssim \|v\|_{Y^{-s_c}}\|w\|_{Y^{s_c}}(\|u\|_{Y^{s_c}} + \|w\|_{Y^{s_c}})^p.
    \end{align*}
    Take $g = w$, $D_{N_1} = \pt_z$, and $h = (P_{\leq \frac{N}{2}} + \theta P_N)(u+w)$.
    Then the first line in the above estimate is very nearly the expression that is estimated in Proposition \ref{prop:incomparable_frequencies_reduced}.
    As we will point out after the proof of Proposition \ref{prop:incomparable_frequencies_reduced}, the projection $P_{\leq \frac{N_1}{2}} + \theta P_{N_1}$ can be replaced by $P_{\leq N_1}$ by a simple argument, and consequently the integral in $\theta$ can be eliminated, leaving us with precisely the expression in Proposition \ref{prop:incomparable_frequencies_reduced}.
    The point is that Proposition \ref{prop:incomparable_frequencies_reduced} is fairly lenient when it comes to choosing the functions $g$ and $h$, in a way which will become apparent during its proof. See Remark \ref{rem:replacing_integrands}.
    We thus obtain
    \begin{align*}
        \sum_{N_0\gg N_1} \int_0^1\bigg(\int_0^T \int_{\bb{T}^3} &|v_{N_0}w_{N_1} \pt_z F((P_{\leq\frac{N_1}{2}} + \theta P_{N_1})(u+w))| ~dxdt\bigg)~d\theta\\
            &\lesssim \|v\|_{Y^{-s_c}}\|w\|_{Y^{s_c}}\max\{\|w\|_{Y^{s_c}}\|u+w\|_{Y^{s_c}}^{p-1},\|u+w\|_{Y^{s_c}}^p\}\\
            &\leq \|v\|_{Y^{-s_c}}\|w\|_{Y^{s_c}}(\|u\|_{Y^{s_c}} + \|w\|_{Y^{s_c}})^p.
    \end{align*}
    Lastly, to completely prove Proposition \ref{prop:incomparable_frequencies} we must obtain the analogous estimate for the remaining five terms in the expression \eqref{eqn:linearization_iteration_1}, and also the corresponding estimates for the sum over $N_0 \ll N_1$.
    A similar argument as above shows how to go from Proposition \ref{prop:incomparable_frequencies_reduced} to the required estimate in each case.
\end{proof}
Before we prove Proposition \ref{prop:incomparable_frequencies_reduced}, let us outline the proof.
The point of isolating the regimes $N_0\gg N_1$ and $N_0\ll N_1$ is that in these cases, we can further restrict the nonlinear factor $D_{N_1}F(h_{\leq N_1})$ to the highest frequencies; e.g. if $N_0\gg N_1$, we can replace $D_{N_1}F(h_{\leq N_1})$ with $P_{\sim N_0}(D_{N_1}F(h_{\leq N_1}))$ inside the integral.
We may then differentiate this term to obtain some extra decay in the highest frequency.
When $2 < p < 4$, the first and second derivatives of $F(z) = |z|^pz$ admit one full derivative, while when $p\geq 4$ they admit two full derivatives.
This produces enough decay to defeat the critical regularity: $s_c < 1$ for $2<p<4$, and $s_c < 2$ for $p\geq 4$.
We will therefore be able to obtain an estimate for 
\[
    \sum_{N_0\gg N_1 \geq 1} \bigg|\int_0^T\int_{\bb{T}^3} v_{N_0}g_{N_1}D_{N_1}F(h_{\leq N_1})~dxdt\bigg| 
\]
which we will be able to sum much like as in the proof of Proposition \ref{prop:main_estimates} for the cubic NLS.

We record some useful estimates before proceeding.
\begin{lemma}[Scaling-critical Strichartz estimate]\label{lem:critical_Strichartz}
    For $p > \frac{4}{3}$,
    \begin{equation}\label{eqn:critical_Strichartz}
        \|u\|_{L_{t,x}^{5p/2}([0,T)\times\bb{T}^3)} \lesssim \|u\|_{Y^{s_c}}.
    \end{equation}
    Also, if $r > \frac{5p}{2}$ with $p>\frac{4}{3}$, then
    \begin{equation}\label{eqn:supercritical_Strichartz}
        \|u_{\leq N}\|_{L_{t,x}^r([0,T)\times\bb{T}^3)} \lesssim N^{\frac{2}{p}-\frac{5}{r}}\|u_{\leq N}\|_{Y^{s_c}}.
    \end{equation}
\end{lemma}
\begin{proof}
    By the square-function estimate,
    \begin{align*}
        \|u\|_{L_{t,x}^{5p/2}}
            \sim \| (\sum_N |u_N|^2)^\frac{1}{2} \|_{L_{t,x}^{5p/2}}
            \leq ( \sum_N \| u_N \|_{L_{t,x}^{5p/2}}^2)^\frac{1}{2}.
    \end{align*}
    Here we have used $p > \frac{4}{3} > \frac{4}{5}$, which ensures that $\|\cdot\|_{L_{t,x}^{5p/4}}$ is a norm and not merely a quasinorm.
    The condition $p>\frac{4}{3}$ also ensures that $\frac{5p}{2} > \frac{10}{3}$.
    Therefore we may apply the Strichartz estimate and obtain
    \[
        ( \sum_N \| u_N \|_{L_{t,x}^{5p/2}}^2)^\frac{1}{2}
            \lesssim ( \sum_N N^{2s_c}\| u_N \|_{Y^0}^2)^\frac{1}{2}\sim \|u\|_{Y^{s_c}}.
    \]
    This proves \eqref{eqn:critical_Strichartz}.
    We proceed similarly for \eqref{eqn:supercritical_Strichartz}: by the square-function estimate, Strichartz, and Cauchy-Schwarz, we obtain
    \begin{align*}
        \|u_{\leq N}\|_{L_{t,x}^r}
            &\lesssim (\sum_{M\leq N} \|u_M\|_{L_{t,x}^r}^2)^\frac{1}{2}
            \lesssim (\sum_{M\leq N} (M^{\frac{2}{p}-\frac{5}{r}})^2\|u_M\|_{Y^{s_c}}^2 )^\frac{1}{2}\\
            &\lesssim N^{\frac{2}{p}-\frac{5}{r}}\|u_{\leq N}\|_{Y^{s_c}}. \qedhere
    \end{align*}
\end{proof}
\begin{lemma}\label{lem:Strichartz_derivatives}
    Let $p > 2$ and $0<T\leq 1$. Then
    \begin{equation}\label{eqn:Strichartz_derivatives_1}
        \|\nabla u_{\leq N}\|_{L_{t,x}^{\frac{10p}{p+4}}([0,T)\times\bb{T}^3)} \lesssim N^{\frac{1}{2}}\|u_{\leq N}\|_{Y^{s_c}([0,T))},
    \end{equation}
    \begin{equation}\label{eqn:Strichartz_derivatives_2}
        \|\nabla u_{\leq N}\|_{L_{t,x}^{\frac{20p}{p+8}}} \lesssim N^{\frac{3}{4}}\|u_{\leq N}\|_{Y^{s_c}},
    \end{equation}
    and
    \begin{equation}\label{eqn:Strichartz_derivatives_3}
        \|\Delta u_{\leq N}\|_{L_{t,x}^{\frac{10p}{p+4}}} \lesssim N^{\frac{3}{2}}\|u_{\leq N}\|_{Y^{s_c}}.
    \end{equation}
\end{lemma}
\begin{proof}
   Note that $\frac{10p}{p+4} > \frac{10}{3}$ for $p> 2$. Applying Bernstein, Strichartz, and Cauchy-Schwarz,
    \begin{align*}
        \|\nabla u_{\leq N}\|_{L_{t,x}^{\frac{10p}{p+4}}}
            &\leq \sum_{M\leq N}\|\nabla u_M\|_{L_{t,x}^{\frac{10p}{p+4}}}
            \sim \sum_{M\leq N} M\|u_M\|_{L_{t,x}^{\frac{10p}{p+4}}}
            \lesssim \sum_{M\leq N}M^{\frac{1}{2}}\|u_M\|_{Y^{s_c}}\\
            &\lesssim N^\frac{1}{2}\|u_{\leq N}\|_{Y^{s_c}}.
    \end{align*}
    The other two estimates are proved similarly.
\end{proof}
\begin{remark}
    Since the Strichartz estimate at the $L_{t,x}^{10/3}$ endpoint only holds with a derivative loss, \eqref{eqn:Strichartz_derivatives_1} and \eqref{eqn:Strichartz_derivatives_3} do not hold for $p=2$.
    This is one reason why we have provided a separate argument for the cubic NLS.
\end{remark}
\begin{lemma}\label{lem:derivative_estimates}
    Let $u\in C^1(\bb{T}^d)$.
    If $p\geq 2$ and $G\in\{\pt_zF,\pt_{\ol{z}}F\}$, then
    \[
        |\nabla (G(u(x)))| \lesssim_p |u(x)|^{p-1}|\nabla u(x)|.
    \]
    and
    \[
        |\Delta G(u(x))| \lesssim_p |u(x)|^{p-1}|\Delta u(x)| + |u(x)|^{p-2}|\nabla u(x)|^2. 
    \]
    If $p\geq 3$ and $G\in\{\pt_z^2F,\pt_z\pt_{\ol{z}}F,\pt_{\ol{z}}^2F\}$, then
    \[
        |\nabla (G(u(x)))| \lesssim_p |u(x)|^{p-2}|\nabla u(x)|
    \]
    and
    \[
        |\Delta (G(u(x)))| \lesssim_p |g(x)|^{p-2}|\Delta u(x)| + |u(x)|^{p-3}|\nabla u(x)|^2.    
    \]
\end{lemma}
The proof of Lemma \ref{lem:derivative_estimates} is a straightforward calculus exercise that we omit.
\begin{proof}[Proof of Proposition \ref{prop:incomparable_frequencies_reduced}]
    Our proof splits along the cases $2 < p < 4$ and $p\geq 4$.\\
    \tbf{Case 1.1:}  ($2 < p < 4$, $N_0 \gg N_1 \geq 1$).
    For fixed $N_0$, we may write
    \begin{align*}
        \bigg|\int_0^T\int_{\bb{T}^3} v_{N_0}g_{N_1}D_{N_1}F(h_{\leq N_1})~dxdt\bigg|
            = \bigg|\int_0^T\int_{\bb{T}^3} v_{N_0}g_{N_1}P_{\sim N_0}(D_{N_1}F(h_{\leq N_1}))~dxdt\bigg|.   
    \end{align*}
    If $g = w$, then $D_{N_1}F = G\in\{\pt_zF,\pt_{\ol{z}}F\}$. Therefore by Bernstein, Lemma \ref{lem:derivative_estimates}, H\"older, Lemma \ref{lem:critical_Strichartz}, and Lemma \ref{lem:Strichartz_derivatives}, we have
    \begin{align*}
        \|P_{\sim N_0}(D_{N_1}F(h_{\leq N_1}))\|_{L_{t,x}^2}
            &\lesssim N_0^{-1}\|\nabla G(h_{\leq N_1})\|_{L_{t,x}^2}\\
            &\lesssim N_0^{-1}\||h_{\leq N_1}|^{p-1}\nabla h_{\leq N_1}\|_{L_{t,x}^2}\\
            &\leq N_0^{-1}\|h_{\leq N_1}\|_{L_{t,x}^{5p/2}}^{p-1}\|\nabla h_{\leq N_1}\|_{L_{t,x}^{\frac{10p}{p+4}}}\\
            &\lesssim N_0^{-1}N_1^{\frac{1}{2}}\|h_{\leq N_1}\|_{Y^{s_c}}^p.
    \end{align*}
    If $g = u$, then $D_{N_1}F = w_{\leq N_1}G$ where $G\in\{\pt_z^2F,\pt_z\pt_{\ol{z}}F,\pt_{\ol{z}}^2F\}$.
    In this case $P_{\sim N_0}(D_{N_1}F(h_{\leq N_1})) = w_{\leq N_1}P_{\sim N_0}(G(h_{\leq N_1}))$, and similarly to above we have
    \begin{align*}
        \|P_{\sim N_0}(D_{N_1}F(h_{\leq N_1}))\|_{L_{t,x}^2}
            &= \|w_{\leq N_1}P_{\sim N_0}(G(h_{\leq N_1}))\|_{L_{t,x}^2}\\
            &\leq \|w_{\leq N_1}\|_{L_{t,x}^{5p/2}}\|P_{\sim N_0}(G(h_{\leq N_1}))\|_{L_{t,x}^{\frac{10p}{5p-4}}}\\
            &\lesssim N_0^{-1}\|w_{\leq N_1}\|_{L_{t,x}^{5p/2}}\|\nabla(G(h_{\leq N_1}))\|_{L_{t,x}^{\frac{10p}{5p-4}}}\\
            &\lesssim N_0^{-1}\|w_{\leq N_1}\|_{L_{t,x}^{5p/2}}\||h_{\leq N_1}|^{p-2}\nabla h_{\leq N_1}\|_{L_{t,x}^{\frac{10p}{5p-4}}}\\
            &\lesssim N_0^{-1}\|w_{\leq N_1}\|_{L_{t,x}^{5p/2}}\||h_{\leq N_1}\|_{L_{t,x}^{5p/2}}^{p-2}\|\nabla h_{\leq N_1}\|_{L_{t,x}^{\frac{10p}{p+4}}}\\
            &\lesssim N_0^{-1}N_1^{\frac{1}{2}}\|w_{\leq N_1}\|_{Y^{s_c}}\|h_{\leq N_1}\|_{Y^{s_c}}^{p-1}.
    \end{align*}
    We now estimate using H\"older, the bilinear Strichartz estimate (Lemma \ref{lem:bilinear_Strichartz}), and the above estimates.
    When $g = w$ we obtain
    \begin{align*}
        &\bigg|\int_0^T\int_{\bb{T}^3} v_{N_0}g_{N_1}P_{\sim N_0}(D_{N_1}F(h_{\leq N_1}))~dxdt\bigg|\\
            &\lesssim \|v_{N_0}g_{N_1}\|_{L_{t,x}^2}\|P_{\sim N_0}(D_{N_1}F(h_{\leq N_1}))\|_{L_{t,x}^2}\\
            &\lesssim \frac{N_1}{N_0}\|v_{N_0}\|_{Y^0}\|w_{N_1}\|_{Y^0}\|h_{\leq N_1}\|_{Y^{s_c}}^p\\
            &\lesssim \left(\frac{N_1}{N_0}\right)^{1-s_c}\|v_{N_0}\|_{Y^{-s_c}}\|w_{N_1}\|_{Y^{s_c}}\|h_{\leq N_1}\|_{Y^{s_c}}^p,
    \end{align*}
    while when $g = u$ we similarly obtain
    \begin{align*}
        &\bigg|\int_0^T\int_{\bb{T}^3} v_{N_0}g_{N_1}P_{\sim N_0}(D_{N_1}F(h_{\leq N_1}))~dxdt\bigg|\\
            &\lesssim \left(\frac{N_1}{N_0}\right)^{1-s_c}\|v_{N_0}\|_{Y^{-s_c}}\|u_{N_1}\|_{Y^{s_c}}\|w_{\leq N_1}\|_{Y^{s_c}}\|h_{\leq N_1}\|_{Y^{s_c}}^{p-1}.
    \end{align*}
    Since $2<p<4$, we have $1-s_c>0$. Therefore the above estimate is summable over $N_0 \gg N_1$ using Cauchy-Schwarz similarly to the proof of \eqref{eqn:reduced_main_estimates} for $p=2$ from Section 3.1.
    Performing the sum, this part of Proposition \ref{prop:incomparable_frequencies_reduced} follows.
    \\
    \tbf{Case 1.2:} ( $2 < p < 4$, $1 \leq N_0 \ll N_1$).
    This case is similar to the previous case.
    Since the highest frequency is $N_1$, we now have
    \begin{align*}
        \bigg|\int_0^T\int_{\bb{T}^3} v_{N_0}g_{N_1}D_{N_1}F(h_{\leq N_1})~dxdt\bigg|
            = \bigg|\int_0^T\int_{\bb{T}^3} v_{N_0}g_{N_1}P_{\sim N_1}(D_{N_1}F(h_{\leq N_1}))~dxdt\bigg|.   
    \end{align*}
    Estimating as before, when $g = w$ we have
    \begin{align*}
        \|P_{\sim N_1}(D_{N_1}F(h_{\leq N_1}))\|_{L_{t,x}^2}
            &\lesssim N_1^{-\frac{1}{2}}\|h_{\leq N_1}\|_{Y^{s_c}}^p
    \end{align*}
    while when $g= u$ we have
    \begin{align*}
        \|P_{\sim N_1}(D_{N_1}F(h_{\leq N_1}))\|_{L_{t,x}^2}
            &\lesssim N_1^{-\frac{1}{2}}\|w_{\leq N_1}\|_{Y^{s_c}}\|h_{\leq N_1}\|_{Y^{s_c}}^{p-1}.
    \end{align*}
    Applying H\"older, bilinear Strichartz, and the above estimates, when $g = w$ we obtain
    \begin{align*}
        &\bigg|\int_0^T\int_{\bb{T}^3} v_{N_0}g_{N_1}P_{\sim N_0}(D_{N_1}F(h_{\leq N_1}))~dxdt\bigg|\\
            &\lesssim \|v_{N_0}g_{N_1}\|_{L_{t,x}^2}\|P_{\sim N_0}(D_{N_1}F(h_{\leq N_1}))\|_{L_{t,x}^2}\\
            &\lesssim \left(\frac{N_0}{N_1}\right)^\frac{1}{2}\|v_{N_0}\|_{Y^0}\|w_{N_1}\|_{Y^0}\|h_{\leq N_1}\|_{Y^{s_c}}^p\\
            &\lesssim \left(\frac{N_0}{N_1}\right)^{\frac{1}{2}+s_c}\|v_{N_0}\|_{Y^{-s_c}}\|w_{N_1}\|_{Y^{s_c}}\|h_{\leq N_1}\|_{Y^{s_c}}^p,
    \end{align*}
    while when $g = u$ we obtain
    \begin{align*}
        &\bigg|\int_0^T\int_{\bb{T}^3} v_{N_0}g_{N_1}P_{\sim N_0}(D_{N_1}F(h_{\leq N_1}))~dxdt\bigg|\\
            &\lesssim \left(\frac{N_0}{N_1}\right)^{\frac{1}{2}+s_c}\|v_{N_0}\|_{Y^{-s_c}}\|u_{N_1}\|_{Y^{s_c}}\|w_{\leq N_1}\|_{Y^{s_c}}\|h_{\leq N_1}\|_{Y^{s_c}}^{p-1}.
    \end{align*}
    Again, $\frac{1}{2}+s_c>0$, so this is summable over $N_0 \ll N_1$ and this case is proved.
    \\
    \tbf{Case 2.1:} ($p \geq 4$, $N_0 \gg N_1 \geq 1$).
    This proof proceeds similarly to that of \tbf{Case 1.1}, except that instead of a gradient we take a Laplacian.
    Since the highest frequency is $N_0$, we localize the nonlinear factor to frequencies $\sim N_0$.
    When $g = w$, we have (with $G\in\{\pt_zF,\pt_{\ol{z}}F\}$)
    \begin{align*}
        \|P_{\sim N_0}&(D_{N_1}F(h_{\leq N_1}))\|_{L_{t,x}^2}\\
            &\lesssim N_0^{-2}\|\Delta G(h_{\leq N_1})\|_{L_{t,x}^2}\\
            &\lesssim N_0^{-2}\||h_{\leq N_1}|^{p-1}|\Delta h_{\leq N_1}| + |h_{\leq N_1}|^{p-2}|\nabla h_{\leq N_1}|^2\|_{L_{t,x}^2}\\
            &\leq N_0^{-2}(\|h_{\leq N_1}\|_{L_{t,x}^{5p/2}}^{p-1}\|\Delta h_{\leq N_1}\|_{L_{t,x}^{\frac{10p}{p+4}}} + \|h_{\leq N_1}\|_{L_{t,x}^{5p/2}}^{p-2}\|\nabla h_{\leq N_1}\|_{L_{t,x}^{\frac{20p}{p+8}}}^2)\\
            &\lesssim N_0^{-2}N_1^{\frac{3}{2}}\|h_{\leq N_1}\|_{Y^{s_c}}^p,
    \end{align*}
    while when $g = u$ we have (with $G\in\{\pt_z^2F,\pt_z\pt_{\ol{z}}F,\pt_{\ol{z}}^2F\}$)
    \begin{align*}
        \|P_{\sim N_0}&(D_{N_1}F(h_{\leq N_1}))\|_{L_{t,x}^2}\\
            &= \|w_{\leq N_1}P_{\sim N_0}(G(h_{\leq N_1}))\|_{L_{t,x}^2}\\
            &\lesssim N_0^{-2}\|w_{\leq N_1}\|_{L_{t,x}^{5p/2}}\|\Delta G(h_{\leq N_1})\|_{L_{t,x}^\frac{10p}{5p-4}}\\
            &\lesssim N_0^{-2}\|w_{\leq N_1}\|_{Y^{s_c}}\||h_{\leq N_1}|^{p-2}|\Delta h_{\leq N_1}| + |h_{\leq N_1}|^{p-3}|\nabla h_{\leq N_1}|^2\|_{L_{t,x}^2}\\
            &\leq N_0^{-2}\|w_{\leq N_1}\|_{Y^{s_c}}(\|h_{\leq N_1}\|_{L_{t,x}^{5p/2}}^{p-2}\|\Delta h_{\leq N_1}\|_{L_{t,x}^{\frac{10p}{p+4}}} + \|h_{\leq N_1}\|_{L_{t,x}^{5p/2}}^{p-3}\|\nabla h_{\leq N_1}\|_{L_{t,x}^{\frac{20p}{p+8}}}^2)\\
            &\lesssim N_0^{-2}N_1^{\frac{3}{2}}\|w_{\leq N_1}\|_{Y^{s_c}}\|h_{\leq N_1}\|_{Y^{s_c}}^{p-1}.
    \end{align*}
    By H\"older, bilinear Strichartz, and the above estimate, when $g=w$ we obtain
    \begin{align*}
        &\bigg| \int_0^T \int_{\bb{T}^3}  v_{N_0} g_{N_1}P_{\sim N_0}(D_{N_1}F(h_{\leq N_1})) ~dxdt\bigg|\\
            &\lesssim \|v_{N_0}g_{N_1}\|_{L_{t,x}^2}\|\|P_{\sim N_0}(D_{N_1}F(h_{\leq N_1}))\|_{L_{t,x}^2}\\
            &\lesssim \left(\frac{N_1}{N_0}\right)^2\|v_{N_0}\|_{Y^0}\|w_{N_1}\|_{Y^0}\|h_{\leq N_1}\|_{Y^{s_c}}^p\\
            &\lesssim \left(\frac{N_1}{N_0}\right)^{2-s_c}\|v_{N_0}\|_{Y^{-s_c}}\|w_{N_1}\|_{Y^{s_c}}\|h_{\leq N_1}\|_{Y^{s_c}}^p
    \end{align*}
    while when $g = u$ we obtain
    \begin{align*}
        &\bigg| \int_0^T \int_{\bb{T}^3}  v_{N_0} g_{N_1}P_{\sim N_0}(D_{N_1}F(h_{\leq N_1})) ~dxdt\bigg|\\
            &\lesssim \|v_{N_0}g_{N_1}\|_{L_{t,x}^2}\|\|P_{\sim N_0}(D_{N_1}F(h_{\leq N_1}))\|_{L_{t,x}^2}\\
            &\lesssim \left(\frac{N_1}{N_0}\right)^2\|v_{N_0}\|_{Y^0}\|u_{N_1}\|_{Y^0}\|w_{\leq N_1}\|_{Y^{s_c}}\|h_{\leq N_1}\|_{Y^{s_c}}^{p-1}\\
            &\lesssim \left(\frac{N_1}{N_0}\right)^{2-s_c}\|v_{N_0}\|_{Y^{-s_c}}\|u_{N_1}\|_{Y^{s_c}}\|w_{\leq N_1}\|_{Y^{s_c}}\|h_{\leq N_1}\|_{Y^{s_c}}^{p-1}.
    \end{align*}
    Since $p\geq 4$, $2-s_c > 0$ and thus this is summable over $N_0 \gg N_1$.
    \\
    \tbf{Case 2.2:} ($p \geq 4$, $1 \leq N_0 \ll N_1$).
    This is covered by the proof of \tbf{Case 1.2}, since $\frac{1}{2}+s_c>0$ for all $p>2$.
\end{proof}
\begin{remark}\label{rem:replacing_integrands}
    As we have alluded to earlier, these estimates are rather lenient with respect to the precise functions inside the integrals.
    For instance, if in \tbf{Case 1.1} we take $g=w$ and estimate $\|P_{\sim N_0}(D_{N_1}F(h_{\leq \frac{N_1}{2}} + \theta h_{N_1}))\|_{L_{t,x}^2}$ instead of $\|P_{\sim N_0} (D_{N_1}F(h_{\leq N_1})\|_{L_{t,x}^2}$, then since $h_{\leq \frac{N_1}{2}} + \theta h_{N_1} = P_{\leq 2N_1}(h_{\leq \frac{N_1}{2}} + \theta h_{N_1})$ we would find via the same proof that
    \begin{align*}
        \|P_{\sim N_0} (D_{N_1}F(h_{\leq \frac{N_1}{2}} + \theta h_{N_1}))\|_{L_t,x}^2
            &\lesssim N_0^{-1}N_1^{\frac{1}{2}}\|h_{\leq\frac{N_1}{2}} + \theta h_{N_1}\|_{Y^{s_c}}^p\\
            &\lesssim N_0^{-1}N_1^{\frac{1}{2}}(\|h_{\leq\frac{N_1}{2}}\|_{Y^{s_c}} + \|h_{\leq N_1}\|_{Y^{s_c}})^p\\
            &\lesssim N_0^{-1}N_1^{\frac{1}{2}}\|h_{\leq N_1}\|_{Y^{s_c}}^p.
    \end{align*}
    This sort of argument can be used to fill in the remaining gaps in our sketch of the proof of Proposition \ref{prop:incomparable_frequencies} from Proposition \ref{prop:incomparable_frequencies_reduced}.
\end{remark}

\subsection{Controlling sums over comparable frequencies}
In this section we prove Proposition \ref{prop:comparable_frequencies}.
In the regime of comparable frequencies $N_0\sim N_1$, the methods in the previous section cannot be used because there is no way to restrict the nonlinear factor to a specific frequency.
Instead, we aim to recreate the case of $p=2$ as best as possible by iterating the paradifferential linearization technique used earlier.
The precise details of how this is done differ between the cases $p\geq 3$ and $2<p<3$, and we treat them separately.
The difference arises from the regularity of $F(z) = |z|^pz$, which determines how many times we may iterate the linearization.

\subsubsection{The case $p \geq 3$}
Let $p\geq 3$. In this case, $F(z) = |z|^pz$ admits four derivatives, and hence we may iterate our paradifferential linearization process four times.
We begin with the formal expression arising from Proposition \ref{prop:Bony_linearization}:
\begin{align*}
    F(u+w) - F(u)
        &= \sum_{N\geq 1} [F(u_{\leq N} + w_{\leq N}) - F(u_{\leq \frac{N}{2}} + w_{\leq \frac{N}{2}})]\\
        &- \sum_{N\geq 1} [F(u_{\leq N}) - F(u_{\leq \frac{N}{2}})]
\end{align*}
where we interpret this equality in terms of convergence in $L^q$, $1\leq q < \frac{d}{2}$, and in particular weak convergence against continuous functions.
Fixing $N_0$, we throw away the summands with $N \gg N_0$ and $N \ll N_0$, and apply \eqref{eqn:FTC_LP} to obtain
\begin{align*}
    \sum_{N_0\sim N_1\geq 1} & (u_{N_1}+w_{N_1})\int_0^1 \pt_zF((P_{\leq\frac{N_1}{2}} + \theta P_{N_1})(u+w))~d\theta + \tnm{similar terms}\\
    &- \sum_{N_0\sim N_1\geq 1} u_{N_1}\int_0^1 \pt_zF((P_{\leq\frac{N_1}{2}} + \theta P_{N_1})u)~d\theta + \tnm{similar terms}.
\end{align*}
We now apply Proposition \ref{prop:Bony_linearization} and \eqref{eqn:FTC} again to the terms inside the integral: that is, we write
\begin{align*}
    \pt_zF&((P_{\leq \frac{N_1}{2}} + \theta P_N)u)\\
        &= \sum_{N \geq 1} [\pt_zF(P_{\leq N}(P_{\leq \frac{N_1}{2}} + \theta P_N)u) - \pt_zF(P_{\leq\frac{N}{2}}(P_{\leq \frac{N_1}{2}} + \theta P_N)u]\\
        &= \sum_{N\geq 1} P_N(P_{\leq \frac{N_1}{2}} + \theta P_N)u\int_0^1 \pt_z^2F((P_{\leq \frac{N}{2}} + \eta P_N)(P_{\leq \frac{N_1}{2}} + \theta P_N)u)~d\eta\\
        &+ \tnm{simliar terms},
\end{align*}
and substitute these inside the integrals.
Note that $P_N(P_{\leq \frac{N_1}{2}} + \theta P_{N_1}) = 0$ for $N > 2N_1$, and $P_N(P_{\leq \frac{N_1}{2}} + \theta P_{N_1})$ is equivalent to $P_N$ for $N \leq 2N_1$ for the purpose of estimates in the manner that we have explained in our sketch of Proposition \ref{prop:incomparable_frequencies} and Remark \ref{rem:replacing_integrands}.
Similarly, we may treat $(P_{\leq \frac{N}{2}} + \eta P_N)(P_{\leq \frac{N_1}{2}} + \theta P_{N_1})$ as $P_{\leq N}$ for $N \leq 2N_1$ for the purpose of proving Proposition \ref{prop:comparable_frequencies}.

We summarize these ideas and calculations in the notation $\sim$.
For two expressions $A$ and $B$, we say $A \sim B$ if $A$ and $B$ are related by collapsing Littlewood-Paley projections (e.g. $P_N(P_{\leq \frac{N_1}{2}} + \theta P_{N_1}) \sim P_N$ for $N\leq 2N_1$), removing integrals over $[0,1]$, and conjugating factors or derivatives (as is encapsulated in the phrase ``similar terms'' as we have used up to this point).
If $A\sim B$, then $A$ and $B$ should admit the same type of estimates in the manner we have described in the sketch of Proposition \ref{prop:incomparable_frequencies} and Remark \ref{rem:replacing_integrands}.
Thus we may succinctly express our first two iterations of the linearization in the following way:
\begin{align*}
    \sum_{N_0\sim N_1\geq 1} &[F(u_{\leq N_1} + w_{\leq N_1}) - F(u_{\leq \frac{N_1}{2}} + w_{\leq \frac{N_1}{2}})]
    - \sum_{N_0\sim N_1\geq 1} [F(u_{\leq N_1}) - F(u_{\leq \frac{N_1}{2}})]\\
    &\sim \sum_{N_0\sim N_1\geq 1} (u_{N_1} + w_{N_1})\pt_zF(u_{\leq N_1} + w_{\leq N_1})
        - \sum_{N_0\sim N_1\geq 1} u_{N_1}\pt_zF(u_{\leq N_1})\\
    &\sim \sum_{N_0\sim N_1 \gtrsim N_2 \geq 1} (u_{N_1}+w_{N_1})(u_{N_2}+w_{N_2})\pt_z^2F(u_{\leq N_2} + w_{\leq N_2})\\
    &\hspace{8em}- \sum_{N_0\sim N_1 \gtrsim N_2 \geq 1} u_{N_1}u_{N_2}\pt_z^2F(u_{\leq N_2}).
\end{align*}
From here on we will use the symbol $\sim$ to summarize all calculations involving Proposition \ref{prop:Bony_linearization}, \eqref{eqn:FTC}, and \eqref{eqn:FTC_LP}.
We now linearize with Proposition \ref{prop:Bony_linearization} and \eqref{eqn:FTC_LP} once more, then shift terms and linearize with \eqref{eqn:FTC} one last time to obtain:
\begin{align*}
    &\sum_{N_0\sim N_1 \gtrsim N_2 \geq 1} (u_{N_1}+w_{N_1})(u_{N_2}+w_{N_2})\pt_z^2F(u_{\leq N_2} + w_{\leq N_2})\\
    &\hspace{7em}- \sum_{N_0\sim N_1 \gtrsim N_2 \geq 1} u_{N_1}u_{N_2}\pt_z^2F(u_{\leq N_2})\\
    &\sim \sum_{N_0\sim N_1 \gtrsim N_2 \gtrsim N_3 \geq 1} (u_{N_1}+w_{N_1})(u_{N_2}
    + w_{N_2})(u_{N_3} + w_{N_3})\pt_z^3F(u_{\leq N_3} + w_{\leq N_3})\\
    &\hspace{7em}- \sum_{N_0\sim N_1 \gtrsim N_2 \gtrsim N_3 \geq 1} u_{N_1}u_{N_2}u_{N_3}\pt_z^3F(u_{\leq N_3})\\
    &\sim \sum_{N_0\sim N_1 \gtrsim N_2 \gtrsim N_3 \geq 1} u_{N_1}^{(1)}u_{N_2}^{(2)}u_{N_3}^{(3)}\pt_z^3 F(u_{\leq N_3} + w_{\leq N_3}) \\
    &\hspace{7em}- \sum_{N_0\sim N_1 \gtrsim N_2 \gtrsim N_3 \geq 1} u_{N_1}u_{N_2}u_{N_3}[\pt_z^3F(u_{\leq N_3} + w_{\leq N_3}) - \pt_z^3 F(u_{\leq N_3})]\\
    &\sim \sum_{N_0\sim N_1 \gtrsim N_2 \gtrsim N_3 \geq 1} u_{N_1}^{(1)}u_{N_2}^{(2)}u_{N_3}^{(3)}\pt_z^3 F(u_{\leq N_3} + w_{\leq N_3}) \\
    &\hspace{8em}- \sum_{N_0\sim N_1 \gtrsim N_2 \gtrsim N_3 \geq 1} u_{N_1}u_{N_2}u_{N_3}w_{\leq N_3}\pt_z^4F(u_{\leq N_3} + w_{\leq N_3}).
\end{align*}
Here $u^{(j)}\in\{u,w\}$ with at least one $u^{(j)} = w$.
Note that in the last step of this linearization, we require $p\geq 3$ so that the fourth-order derivatives of $F$ in $z$ and $\ol{z}$ are well-defined.
Therefore to establish Proposition \ref{prop:comparable_frequencies} for $p\geq 3$, it suffices to prove:
\begin{proposition}\label{prop:comparable_frequencies_reduced_p_geq_3}
    Fix $p \geq 3$.
    Let $0 < T \leq 1$.
    Then
    \begin{align*}
        &\sum_{N_0\sim N_1\gtrsim N_2\gtrsim N_3\geq 1} \bigg|\int_0^T\int_{\bb{T}^3} v_{N_0}u_{N_1}^{(1)}u_{N_2}^{(2)}u_{N_3}^{(3)}D_{N_3}F(h_{\leq N_3})~dxdt\bigg|\\
        &\hspace{5em}\lesssim \|v\|_{Y^{-s_c}}\|u^{(1)}\|_{Y^{s_c}}\|u^{(2)}\|_{Y^{s_c}}\|u^{(3)}\|_{Y^{s_c}}\max\{\|w\|_{Y^{s_c}},\|h\|_{Y^{s_c}}\}\|h\|_{Y^{s_c}}^{p-3},
    \end{align*}
    where
    \[
        D_{N_3} \in\{\pt_z^3,\pt_z^2\pt_{\ol{z}},\pt_z\pt_{\ol{z}}^2,\pt_{\ol{z}}^3\}
    \]
    if $u^{(j)} = w$ for some $j=1,2,3$, and 
    \[
        D_{N_3} \in \{w_{\leq N_3}\pt_z^4,w_{\leq N_3}\pt_z^3\pt_{\ol{z}},w_{\leq N_3}\pt_z^2\pt_{\ol{z}}^2,,w_{\leq N_3}\pt_z\pt_{\ol{z}}^3,w_{\leq N_3}\pt_{\ol{z}}^4\}
    \]
    if $u^{(j)}\neq w$ for all $j=1,2,3$.
\end{proposition}
\begin{remark}\label{rem:fourth_derivatives}
    The condition $p\geq 3$ is only required for the estimates where $u^{(j)}\neq w$ for $j=1,2,3$, since these require the fourth-order derivative of $F$ to be defined.
    For the estimates where some $u^{(j)} = w$, we need only $p\geq 2$.
    This will be useful in the next section.
\end{remark}
\begin{proof}
    For any integer $0\leq k < p+1$, and any $k$-th order derivative $G$ of $F(z) = |z|^pz$, we have $|G(z)| \lesssim |z|^{p-k}$.
    Therefore
    \[
        |D_{N_3}F(h_{\leq N_3})| \lesssim \max\{|w_{\leq N_3}|,|h_{\leq N_3}|\}|h_{\leq N_3}|^{p-3}.    
    \]
    Therefore by Lemma \ref{lem:critical_Strichartz} we have
    \begin{align*}
        \|D_{N_3}F(h_{\leq N_3})\|_{L_{t,x}^\infty}
            &\lesssim N^{\frac{2(p-2)}{p}}\max\{\|w_{\leq N_3}\|_{Y^{s_c}},\|h_{\leq N_3}\|_{Y^{s_c}}\}\|h_{\leq N_3}\|_{Y^{s_c}}^{p-3}.  
    \end{align*}
    Estimating the integral by H\"older, bilinear Strichartz, and Strichartz, we have
    \begin{align*}
        &\bigg| \int_0^T \int_{\bb{T}^3} v_{N_0}u_{N_1}^{(1)}u_{N_2}^{(2)}u_{N_3}^{(3)}D_{N_3}F(h_{\leq N_3}) ~dxdt \bigg|\\
            &\lesssim \|v_{N_0}u_{N_2}^{(2)}\|_{L_{t,x}^2}\|u_{N_1}^{(1)}u_{N_3}^{(1)}\|_{L_{t,x}^2}\|D_{N_3}F(h_{\leq N_3})\|_{L_{t,x}^\infty}\\
            &\lesssim \left(\frac{N_0}{N_1}\right)^{s_c}\left(\frac{N_3}{N_2}\right)^{s_c-\frac{1}{2}}\|v_{N_0}\|_{Y^{-s_c}}\|u_{N_1}^{(1)}\|_{Y^{s_c}}\|u_{N_2}^{(2)}\|_{Y^{s_c}}\|u_{N_3}^{(3)}\|_{Y^{s_c}}\\
            &\hspace{10em}\cdot\max\{\|w_{\leq N_3}\|_{Y^{s_c}},\|h_{\leq N_3}\|_{Y^{s_c}}\}\|h_{\leq N_3}\|_{Y^{s_c}}^{p-3}.
    \end{align*}
    Summing with Cauchy-Schwarz first over $N_2 \gtrsim N_3$, then over $N_0\sim N_1$ establishes the desired estimate.
\end{proof}

\subsubsection{The case $2<p<3$}
The last case to consider is $2<p<3$.
In this case, $F$ does not admit four derivatives, so we cannot obtain the linearized expression we used for the case $p\geq 3$.
However, $F$ admits three derivatives and the third derivatives of $F$ are H\"older continuous, which we can ``differentiate'' to obtain sufficient decay to sum.

First we obtain the linearization of the nonlinearity.
Employing the notation $\sim$ as in the previous section, we obtain:
\begin{align*}
    &\sum_{N_0\sim N_1\geq 1} [F(u_{\leq N_1} + w_{\leq N_1}) - F(u_{\leq\frac{N_1}{2}} + w_{\leq\frac{N_1}{2}})] - \sum_{N_0\sim N_1\geq 1} [F(u_{\leq N_1}) - F(u_{\leq\frac{N_1}{2}})]\\
        &\sim \sum_{N_0\sim N_1\geq 1} (u_{N_1} +w_{N_1})\pt_zF(u_{\leq N_1} + w_{\leq N_1}) - \sum_{N_0\sim N_1\geq 1} u_{N_1}\pt_zF(u_{\leq N_1}) \\
        &\sim \sum_{N_0\sim N_1\gtrsim N_2\geq 1} (u_{N_1} +w_{N_1})(u_{N_2} + w_{N_2})\pt_z^2F(u_{\leq N_2} + w_{\leq N_2}) \\
        &\hspace{6em}- \sum_{N_0\sim N_1 \gtrsim N_2\geq 1} u_{N_1}u_{N_2}\pt_z^2F(u_{\leq N_2})\\
        &\sim \sum_{N_0\sim N_1\gtrsim N_2\geq 1} u_{N_1}^{(1)}u_{N_2}^{(2)}\pt_z^2F(u_{\leq N_2} + w_{\leq N_2}) \\
        &\hspace{6em}+ \sum_{N_0\sim N_1 \gtrsim N_2\geq 1} u_{N_1}u_{N_2}[\pt_z^2 F(u_{\leq N_2} + w_{\leq N_2}) - \pt_z^2F(u_{\leq N_2})]\\
        &\sim \sum_{N_0\sim N_1\gtrsim N_2\gtrsim N_3\geq 1} u_{N_1}^{(1)}u_{N_2}^{(2)}(u_{N_3} + w_{N_3})\pt_z^3F(u_{\leq N_3} + w_{\leq N_3}) \\
        &\hspace{6em}+ \sum_{N_0\sim N_1 \gtrsim N_2\geq 1} u_{N_1}u_{N_2}w_{\leq N_2}\pt_z^3 F(u_{\leq N_2} + w_{\leq N_2}).
\end{align*}
The first summation in the last line is precisely of a form that is controlled by Proposition \ref{prop:comparable_frequencies_reduced_p_geq_3}; see Remark \ref{rem:fourth_derivatives}.
Writing $w_{\leq N_2} = \sum_{N_2\geq N_3} w_{N_3}$, to establish Proposition \ref{prop:comparable_frequencies} for $2<p<3$, it suffices to prove the following estimate:
\begin{proposition}\label{prop:comparable_frequencies_reduced_p_2_to_3}
    Fix $2<p<3$.
    Let $0 < T \leq 1$.
    Then
    \begin{align*}
        \sum_{N_0\sim N_1\gtrsim N_2\geq N_3 \geq 1} &\bigg|\int_0^T \int_{\bb{T}^3} v_{N_0}u_{N_1}u_{N_2}w_{N_3}G(h_{\leq N_2})~dxdt\bigg|\\
            &\lesssim \|v\|_{Y^{-s_c}}\|u\|_{Y^{s_c}}^2\|w\|_{Y^{s_c}}\|h\|_{Y^{s_c}}^{p-2}
    \end{align*}
    where $G\in\{\pt_z^3F,\pt_z^2\pt_{\ol{z}}F,\pt_z\pt_{\ol{z}}^2F,\pt_{\ol{z}}^3F\}$.
\end{proposition}
\begin{proof}
    Proposition \ref{prop:comparable_frequencies_reduced_p_2_to_3} follows from the following two estimates:
    \begin{align*}\label{eqn:low_frequency_estimate}
        \sum_{N_0 \sim N_1 \geq N_2 \geq N_3} &\bigg| \int_0^T\int_{\bb{T}^3} v_{N_0}u_{N_1}u_{N_2}w_{N_3}P_{\leq N_2}(G(h_{\leq N_2}))~dxdt \bigg|\\
           &\lesssim \|v\|_{Y^{-s_c}}\|u\|_{Y^{s_c}}^2\|w\|_{Y^{s_c}}\|h\|_{Y^{s_c}}^{p-2}, \numberthis
    \end{align*}
    \begin{align*}\label{eqn:high_frequency_estimate}
        \sum_{N_0 \sim N_1 \geq N_2 \geq N_3} \sum_{N > N_2} &\bigg| \int_0^T\int_{\bb{T}^3} v_{N_0}u_{N_1}u_{N_2}w_{N_3}P_N(G(h_{\leq N_2}))~dxdt \bigg|\\
           &\lesssim \|v\|_{Y^{-s_c}}\|u\|_{Y^{s_c}}^2\|w\|_{Y^{s_c}}\|h\|_{Y^{s_c}}^{p-2}. \numberthis
    \end{align*}
    \tbf{Proof of \eqref{eqn:low_frequency_estimate}:}
    For a given $N_2$, let $\bb{Z}^3 = \bigcup_j C_j$ be a partition of frequency space into cubes $C_j$ of side length $N_2$.
    We write $C_j\sim C_k$ if the sum set $C_j + C_k$ intersects the Fourier support of $P_{\leq 3N_2}$.
    For a given $C_j$, there are finitely many $C_k$ with $C_j\sim C_k$, and the number of such $C_k$ is uniformly bounded independently of $N_2$.

    To prove \eqref{eqn:low_frequency_estimate} it then suffices to prove
    \begin{align*}
        \sum_{N_0 \sim N_1 \geq N_2 \geq N_3} \sum_{C_j\sim C_k} &\bigg| \int_0^T\int_{\bb{T}^3} (P_{C_j}v_{N_0})(P_{C_k}u_{N_1})u_{N_2}w_{N_3}P_{\leq N_2}(G(h_{\leq N_2}))~dxdt \bigg|\\
           &\lesssim \|v\|_{Y^{-s_c}}\|u\|_{Y^{s_c}}^2\|w\|_{Y^{s_c}}\|h\|_{Y^{s_c}}^{p-2}.
    \end{align*}
    First let us proceed formally.
    Note that $|G(z)| \lesssim |z|^{p-2}$.
    By H\"older and Strichartz, we have
    \begin{align*}
        &\bigg| \int_0^T\int_{\bb{T}^3} (P_{C_j}v_{N_0})(P_{C_k}u_{N_1})u_{N_2}w_{N_3}P_{\leq N_2}(G(h_{\leq N_2}))~dxdt \bigg|\\
            &\leq \|P_{C_j}v_{N_0}\|_{L_{t,x}^{r_0}}\|P_{C_k}u_{N_1}\|_{L_{t,x}^{r_0}}\|u_{N_2}\|_{L_{t,x}^{r_0}}\|w_{N_3}\|_{L_{t,x}^{r_1}}\|P_{\leq N_2}(G(h_{\leq N_2}))\|_{L_{t,x}^\infty}\\
            &\lesssim \left(\frac{N_0}{N_1}\right)^{s_c}\left(\frac{N_3}{N_2}\right)^{\frac{3}{2}-\frac{5}{r_1}-s_c}\|P_{C_j}v_{N_0}\|_{Y^{-s_c}}\|P_{C_k}u_{N_1}\|_{Y^{s_c}}\|u_{N_2}\|_{Y^{s_c}}\|w_{N_3}\|_{Y^{s_c}}\|h_{\leq N_2}\|_{Y^{s_c}}^{p-2},
    \end{align*}
    provided that $r_j$ are H\"older exponents with $r_j>\frac{10}{3}$, $j=0,\ldots,4$.
    The lowest frequency is summable if $r_1 > \frac{5p}{2}$.
    We take $r_0 = \frac{15p}{5p-2(1-\ep)}$ and $r_1 = \frac{5p}{2(1-\ep)}$.
    Then for $\ep = \ep(p)$ sufficiently small, we have $r_0,r_1 > \frac{10}{3}$.
    Summing using Cauchy-Schwarz, \eqref{eqn:low_frequency_estimate} follows.
    \\
    \tbf{Proof of \eqref{eqn:high_frequency_estimate}:}
    As before, let $\bb{Z}^3 = \bigcup_j C_j$ be a partition of frequency space into cubes, except that the $C_j$ now have side length $N$ and $C_j\sim C_k$ if $C_j+C_k$ intersects the Fourier support of $P_{\leq 3N}$.
    Then like the previous proof, \eqref{eqn:high_frequency_estimate} follows from
    \begin{align*}
        \sum_{N_0 \sim N_1 \geq N_2 \geq N_3}\sum_{N > N_2} \sum_{C_j\sim C_k} &\bigg| \int_0^T\int_{\bb{T}^3} (P_{C_j}v_{N_0})(P_{C_k}u_{N_1})u_{N_2}w_{N_3}P_N(|h_{\leq N_2}|^{p-2})~dxdt \bigg|\\
           &\lesssim \|v\|_{Y^{-s_c}}\|u\|_{Y^{s_c}}^2\|w\|_{Y^{s_c}}\|h\|_{Y^{s_c}}^{p-2}.
    \end{align*}
    The new ingredient relative to the preceding is the following estimate: for $2<p<3$ and $r > \frac{10}{3}$,
    \begin{equation}\label{eqn:Bernstein_high_frequencies}
        \|P_N (G(h_{\leq N_2}))\|_{L_{t,x}^{r/(p-2)}}
            \lesssim N^{-(p-2)}N_2^{(\frac{5}{2}-\frac{5}{r_4}-s_c)(p-2)}\|u_{\leq N_2}\|_{Y^{s_c}}^{p-2}.
    \end{equation}
    This estimate follows from Lemma \ref{prop:nonlinear_Bernstein} and the square-function estimate.
    Choosing $r_0 = r_1 = \frac{20p}{(1-\ep)p^2 + (1+5\ep)p + 4\ep}$, $r_2 = \frac{10p}{2p^2 - 4 - 3(1-\frac{\ep}{3})p(p-2)}$, $r_3 = \frac{5p}{2(1-\ep)}$, and $r_4 = \frac{10}{3(1-\ep)}$, and taking $\ep = \ep(p)>0$ small,
    we find that $r_j > \frac{10}{3}$ for $j=0,1,2,3,4$.
    Proceeding as above, by H\"older, Bernstein, and \eqref{eqn:Bernstein_high_frequencies} we obtain
    \begin{align*}
        &\bigg| \int_0^T\int_{\bb{T}^3} (P_{C_j}v_{N_0})(P_{C_k}u_{N_1})u_{N_2}w_{N_3}P_N(G(h_{\leq N_2}))~dxdt \bigg|\\
        &\lesssim \|P_{C_j}v_{N_0}\|_{L_{t,x}^{r_0}}\|P_{C_k}u_{N_1}\|_{L_{t,x}^{r_1}}\|u_{N_2}\|_{L_{t,x}^{r_2}}\|w_{N_3}\|_{L_{t,x}^{r_3}}\|P_N(G(h_{\leq N_2}))\|_{L_{t,x}^{r_4/(p-2)}}\\
        &\lesssim \left(\frac{N_0}{N_1}\right)^{s_c}N^{5(\frac{1}{r_2} + \frac{1}{r_3} + \frac{p-2}{r_4})-p}N_2^{-5(\frac{1}{r_2} + \frac{p-2}{r_4})+p-\frac{2}{p}} N_3^{\frac{2}{p}-\frac{5}{r_3}}\\
        &\hspace{4em}\cdot \|P_{C_j}v_{N_0}\|_{Y^{-s_c}}\|P_{C_k}u_{N_1}\|_{Y^{s_c}}\|u_{N_2}\|_{Y^{s_c}}\|w_{N_3}\|_{Y^{s_c}}\|h_{\leq N_2}\|_{Y^{s_c}}^{p-2}.
    \end{align*}
    Our choices of $r_j$ also ensure that this is summable over $N_0\sim N_1\gtrsim N_2\gtrsim N_3$, $N>N_2$.
    Performing the summation establishes \eqref{eqn:high_frequency_estimate}.
\end{proof}

\section*{Acknowledgments} The author thanks his advisors Rowan Killip and Monica Vi\c san for many helpful discussions and guidance.


\end{document}